\documentclass[preprint,12pt]{elsarticle}

\usepackage{amssymb}
\newtheorem{thm}{Theorem}

\newdefinition{dfn}{Definition}
\newdefinition{ex}{Example}
\newproof{proof}{Proof}
\newtheorem{cor}{Corollary}
\newtheorem{lem}{Lemma}
\newtheorem{rem}{Remark}

\begin{document}

\begin{frontmatter}



%\title{GROUP $SU$-ACTION AND ITS APPLICATIONS TO GROUP THEORY}
\title{Selfadjoint
realization  of boundary-value problems with interior singularities}

%% \address[label1]
\author[rvt]{K. Aydemir}
\ead{kadriye.aydemir@gop.edu.tr}
\author[rvt]{O. Sh. Mukhtarov\corref{cor1}}
\ead{omukhtarov@yahoo.com} \cortext[cor1]{Corresponding Author (Tel:
+90 356 252 16 16, Fax: +90 356 252 15 85)}

%%\author[rvt]{H. Ol\v{g}ar}
%%\ead{hayati.olgar@gop.edu.tr}

\address[rvt]{Department of Mathematics, Faculty of Arts and Science, Gaziosmanpa\c{s}a University,\\
 60250 Tokat, Turkey}
%%\address[label2]{Department of Mathematics, Faculty of Arts and Science, Gaziosmanpa\c{s}a University,\\
 %60250 Tokat, Turkey}
%%\address[label3]{Department of Mathematics, Faculty of Arts and Science, Gaziosmanpa\c{s}a University,\\
 %60250 Tokat, Turkey}

%----------------------------------------------------------------------

\begin{abstract}
The purpose of this paper is to  investigate some spectral
properties of Sturm-Liouville type problems with interior
singularities. Some of the mathematical aspects necessary for
developing  own technique presented. By applying this technique we
construct some special solutions of the homogeneous equation and
present a formula and the existence conditions of Green's function.
Further based on this results and introducing  operator treatment in
adequate Hilbert space  we derive the resolvent  operator and prove
selfadjointness of the considered problem.
\end{abstract}


\end{frontmatter}

% \linenumbers

%------------------------INTRODUCTION---------------------------
\section{Introduction}
%---------------------------------------------------------------
For inhomogeneous linear systems, the basic Superposition Principle
says that the response to a combination of external forces is the
self-same combination of responses to the individual forces. In a
finite-dimensional system, any forcing function can be decomposed
into a linear combination of unit impulse forces, each applied to a
single component of the system, and so the full solution can be
written as a linear combination of the solutions to the impulse
problems. This simple idea will be adapted to boundary value
problems governed by differential equations, where the response of
the system to a concentrated impulse force is known as the Green's
function. With the Green's function in hand, the solution to the
inhomogeneous system with a general forcing function can be
reconstructed by superimposing the effects of suitably scaled
impulses. In this study we shall investigate some spectral
properties of the Sturm-Liouville differential equation on two
interval
\begin{equation}\label{1}
\mathcal{L}y:=-y^{\prime \prime }(x)+ q(x)y(x)=\lambda y(x), \ \ \ \
x\in [a,c)\cup(c,b]
\end{equation}
on  $[a, c)\cup(c,b]$, with eigenparameter- dependent boundary
conditions at the end points $x=a$ and $x=b$
\begin{equation}\label{2}
 \tau_{1}(y):=\alpha_{10}y(a)+\alpha_{11}y'(a)=0,
\end{equation}
\begin{equation}\label{3}
\tau_{2}(y):=\alpha_{20}y(b)-\alpha_{21}y'(b)+\lambda(\alpha'_{20}y(b)-\alpha'_{21}y'(b))=0
\end{equation}  and the transmission conditions at the singular interior point $x=c$
\begin{equation}\label{4}
\tau_{3}(y):=\beta^{-}_{11}y'(c-)+\beta^{-}_{10}y(c-)+\beta^{+}_{11}y'(c+)+\beta^{+}_{10}y(c+)=0,
\end{equation}
\begin{equation}\label{5}
\tau_{4}(y):=\beta^{-}_{21}y'(c-)+\beta^{-}_{20}y(c-)+\beta^{+}_{21}y'(c+)+\beta^{+}_{20}y(c+)=0,
\end{equation}
where the potential  $q(x)$ is real continuous function in each of
the intervals $[a, c) \ \textrm{and} (c, b]$ and has a finite limits
$q( c\mp0)$, $\lambda$ \ is a complex spectral parameter, \
$\alpha_{ij}, \ \ \beta^{\pm}_{ij}, \ (i=1,2 \
 \textrm{and} \ \ j=0,1), \alpha'_{ij}(i=2 \ \textrm{and} \ j=0,1)$ are real
numbers.

Our problem differs from the usual regular Sturm-Liouville problem
in the sense that the eigenvalue parameter $\lambda$ are contained
in both differential equation and boundary conditions and two
supplementary transmission conditions at one interior point are
added to boundary conditions.  Such problems are connected with
discontinuous material properties, such as heat and mass transfer,
vibrating string problems when the string loaded additionally with
points masses, diffraction problems \cite{ li, tit} and varied
assortment of physical transfer problems. We develop an own
technique for investigation some spectral properties of this
problem. In particular, we construct the Green function's and
adequate Hilbert space for selfadjoint realization of the considered
problem.

%------------------------PRELIMINARIES---------------------------
\section{\textbf{Some basic solutions and
Green's function}}Denote the determinant of the k-th and j-th
columns of the matrix
$$T= \left[%
\begin{array}{cccc}
  \beta^{-}_{10} & \beta^{-}_{11} & \beta^{+}_{10} & \beta^{+}_{11} \\
  \beta^{-}_{20} & \beta^{-}_{21} & \beta^{+}_{20} & \beta^{+}_{21}
  \\
\end{array} %
 \right] $$
 by $\Delta_{kj}$. For selfadjoint realization in adequate Hilbert space, everywhere in below we shall assume that
  \begin{equation} \label{se}  \Delta_{12}>0 \ \textrm{and} \ \Delta_{34}>0.\end{equation}
 With a view to constructing the Green's  function we shall define  two special solution of the equation
  (\ref{1})  by own technique as follows.
At first consider the next initial-value problem on the left
interval $[a,c)$
\begin{equation}
-y^{\prime \prime }+q(x)y=\lambda y
  \label{(6)}
\end{equation}
\begin{equation}\label{7}
y(a)=\alpha _{11}, \ y^{\prime }(a)=-\alpha _{10}.
\end{equation}
It is known that this problem has an unique solution \
$u=\varphi^{-}(x,\lambda)$ \ which is an entire function of $\lambda
\in \mathbb{C}$ for each fixed $x \in [a,c)$ (see, for example,
\cite{ti}). By applying the similar method of \cite{os1}  we can
prove that  the equation (\ref{1}) on the right interval  $(c,b]$
has an unique solution $u=\varphi^{+}(x,\lambda)$ satisfying the
equalities
\begin{equation}\label{8}
\varphi^{+}(c+,\lambda )
=\frac{1}{\Delta_{12}}(\Delta_{23}\varphi^{-}(c-,\lambda
)+\Delta_{24}\frac{\partial\varphi^{-}(c-,\lambda )}{\partial x})
\end{equation}
\begin{equation}\label{9}
\frac{\partial\varphi^{+}(c+,\lambda )}{\partial x})
=\frac{-1}{\Delta_{12}}(\Delta_{13}\varphi ^{-}(c-,\lambda
)+\Delta_{14}\frac{\partial\varphi^{-}(c-,\lambda )}{\partial x}).
\end{equation}
which also is an entire function of the parameter $\lambda $ for
each fixed $x \in [c,b]$. Consequently, the solution
$u=\varphi(x,\lambda)$ defined by
\begin{eqnarray} \varphi(x,\lambda)=\left
\{\begin{array}{ll}\label{91}
\varphi^{-}(x,\lambda), & x\in \lbrack a,c) \\
\varphi^{+}(x,\lambda), & x\in (c,b\rbrack \\
\end{array}\right.
\end{eqnarray}
 satisfies the equation (\ref{1}) on whole $[a,c) \cup (c,b]$, the
 first boundary condition of (\ref{2}) and both transmission
 conditions(\ref{4}) and (\ref{5}).\\
 By the same technique, we can define the solution
\begin{eqnarray} \psi(x,\lambda)=\left
\{\begin{array}{ll}\label{92}
\psi^{-}(x,\lambda), & x\in \lbrack a,c) \\
\psi^{+}(x,\lambda), & x\in (c,b\rbrack  \\
\end{array}\right.
\end{eqnarray}
so that
\begin{equation}\label{10}
\psi(b,\lambda)=\alpha _{21}+\lambda\alpha' _{21}, \
\frac{\partial\psi(b,\lambda )}{\partial x})=\alpha
_{20}+\lambda\alpha' _{20}
\end{equation}
and
 \begin{equation}\label{11}
\psi^{-}(c-,\lambda )
=\frac{-1}{\Delta_{34}}(\Delta_{14}\psi^{+}(c+,\lambda
)+\Delta_{24}\frac{\partial\psi^{+}(c+,\lambda )}{\partial x}),
\end{equation}
\begin{equation}\label{12}
\frac{\partial\psi^{+}(c-,\lambda )}{\partial x})
=\frac{1}{\Delta_{34}}(\Delta_{13}\psi ^{+}(c+,\lambda
)+\Delta_{23}\frac{\partial\psi^{+}(c+,\lambda )}{\partial x}).
\end{equation}
Consequently, $\psi(x,\lambda)$  satisfies the equation (\ref{1}) on
whole $[a,c) \cup (c,b]$, the second boundary condition  (\ref{3})
and both transmission condition  (\ref{4}) and (\ref{5}). By using
(\ref{8}),(\ref{9}), (\ref{11}) and (\ref{12}) and the well-known
fact that the Wronskians $W[\varphi^{-}(x,\lambda ),\psi
^{-}(x,\lambda )]$ and $W[\varphi ^{+}(x,\lambda ),\psi
^{+}(x,\lambda )]$ are independent of variable $x$ it is easy to
show that $ \Delta_{12}w^{+}(\lambda ) = \Delta_{34} w^{-}(\lambda
).$ We shall introduce  the characteristic function for the problem
$(\ref{1})-(\ref{5})$ as
\begin{eqnarray*}\label{13}
w(\lambda):= \Delta_{34} w^{-}(\lambda) = \Delta_{12} \
w^{+}(\lambda).
\end{eqnarray*}
Similarly to \cite{os1} we can prove that, there are infinitely many
eigenvalues $\lambda_{n}, \ n=1,2,...$   of the BVTP
$(\ref{1})-(\ref{5})$ which are coincide with
the zeros of characteristic function  \ $w(\lambda)$.\\
Now, let us consider the non-homogenous differential equation
\begin{eqnarray}\label{141}
y''+(\lambda-q(x))y=f(x), \
\end{eqnarray}
on $[a,c)\cup(c,b]$ together with the same boundary and transmission
conditions $(\ref{1})-(\ref{5})$, when $w(\lambda)\neq0.$ The
following formula is obtained for the solution $Y=Y_{0}(x,\lambda )$
of the equation $(\ref{141})$ under boundary and transmission
conditions $(\ref{2})$-$(\ref{5})$
\begin{eqnarray}\label{e8}
Y_{0}(x,\lambda)=\left\{\begin{array}{c}
               \frac{\Delta_{34}\psi^{-}(x,\lambda)}{\omega(\lambda)}\int_{a}^{x}\varphi^{-}(y,\lambda)f(y)dy +
\frac{\Delta_{34}\varphi^{-}(x,\lambda)}{\omega(\lambda)}\int_{x}^{c-}\psi^{-}(y,\lambda)f(y)dy
\\ +\frac{\Delta_{12}\varphi^{-}(x,\lambda)}{w(\lambda
)}\int\limits_{c+}^{b}f(y)\psi
^{+}(y,\lambda)dy \ , \ \ \ \ \ \ \ \ for \  x \in [a,c) \\
                \\
               \frac{\Delta_{12}\psi^{+}(x,\lambda)}{\omega(\lambda)}\int_{c+}^{x}\varphi^{+}(y,\lambda)f(y)dy +
\frac{\Delta_{12}\varphi^{+}(x,\lambda)}{\omega(\lambda)}\int_{x}^{b}\psi^{+}(y,\lambda)f(y)dy\\
+\frac{\Delta_{34}\psi^{+}(x,\lambda)}{w(\lambda
)}\int\limits_{a}^{c-}f(y)\varphi ^{-}(y,\lambda)dy \ , \ \ \ \ \ \ \ \ for \ x \in (c,b] \\
             \end{array}\right.
\end{eqnarray}
 From this formula we find that the Green function's of the problem
 $(\ref{1})$-$(\ref{5})$
has the form
\begin{eqnarray}\label{23}
G_{0}(x,y;\lambda)=\left\{\begin{array}{c}
 \frac{\varphi(y,\lambda)\psi(x,\lambda)}{\omega(\lambda)} \ \ \ \ \textrm{for} \  a\leq y\leq x\leq b, \ \  x, y \neq c\\
                \\
  \frac{\varphi(x,\lambda)\psi(y,\lambda)}{\omega(\lambda)} \ \ \ \ \textrm{for} \  a\leq x\leq y\leq b, \ \  x, y\neq c \\
             \end{array}\right.
\end{eqnarray}
and the solution  $(\ref{e8})$  can be rewritten in the terms of
this Green function's as
\begin{eqnarray}\label{ye} Y_{0}(x,\lambda
)=\Delta_{12}\int\limits_{a}^{c-}G_{0}(x,y;\lambda
)f(y)dy+\Delta_{34}\int\limits_{c+}^{b}G_{0}(x,y;\lambda )f(y)dy.
\end{eqnarray}
\section{\textbf{Construction of the   Resolvent operator by means of\\ Green's function the in adequate Hilbert space} }
In this section we define a  linear operator A in suitable Hilbert
space  such a way that the considered problem can be interpreted as
the eigenvalue problem of this operator. For this we assume that
$\Delta_{0}:= \alpha_{21}\alpha'_{20} - \alpha_{20} \alpha'_{21}>0$
and introduce a new  inner product in the Hilbert space
$H=(L_{2}[a,c)\oplus L_{2}(c,b])\oplus \mathbb{C}$   by
$$
<F,G>_{1}:=\Delta_{12} \ \int_{a}^{c-} f(x)\overline{g(x)}dx +
\Delta_{34} \int_{c+}^{b} f(x)\overline{g(x)}dx +
\frac{\Delta_{34}}{\Delta_{0}}f_{1}\overline{g_{1}}
$$
for $F=\left(
  \begin{array}{c}
   f(x),
    f_{1} \\
  \end{array}
\right)$,\quad $G=\left(
  \begin{array}{c}
    g(x),
  g_{1} \\
  \end{array}
\right)\in H.$
\begin{rem}
Note that this  modified inner product is equivalent to  standard
inner product of $(L_{2}[a,c)\oplus L_{2}(c,b])\oplus\mathbb{C}$, so
$H_{1}=(L_{2}[a,c)\oplus L_{2}(c,b] \oplus \mathbb{C}\\,<.,.>_{1} )$
is also Hilbert space.
\end{rem} For convenience denote
$$T_{b}(f):= \alpha_{20}f(b)-\alpha_{21}f'(b),  \ \  T'_{b}(f):= \alpha'_{20}f(b)-\alpha'_{21}f'(b)$$
and define a linear operator
$$A(\mathcal{L}f(x),
T'_{b}(f))=(\mathcal{L} f, -T_{b}(f))$$ with the domain $D(A)$
consisting of all elements $(f(x), f_{1}) \in H_{1},$ such that
$f(x) \ \textrm{and} \  f'(x)$ are absolutely continuous in each
interval [a,c) \ and (c,b],  and has a finite limit $f(c\mp0)
\textrm{and} \ f'_{1}(c\mp0)$,  $\mathcal{L}f$ $\in L_{2}[a,b]$,
$\tau_{1}f=\tau_{3}f=\tau_{4}f=0$ and
$f_{1}=T'_{b}(f).$\\
Consequently the problem $(\ref{1})-(\ref{5})$ can be written in the
operator form as $$AF=\lambda F, \ \ F=(f(x), T'_{b}(f))
 \in D(A)$$
in the Hilbert space $H_{1}$. It is easy to see that, the operator A
is well defined in $H_{1}$.  Let A be defined as above and let
$\lambda$ not be an eigenvalue of this operator. For construction
the resolvent operator $R(\lambda,A):= (\lambda-A)^{-1}$ we shall
solve the operator equation
\begin{eqnarray}\label{y1}
(\lambda-A)Y=F
\end{eqnarray}
for $F\in H_{1}$. This operator equation is equivalent to the
nonhomogeneous differential equation
\begin{eqnarray}\label{14}
y''+(\lambda-q(x))y=f(x), \
\end{eqnarray}
on $[a,c)\cup(c,b]$ subject to nonhomogeneous boundary conditions
and homogeneous transmission conditions
\begin{equation}\label{15}
\tau_{1}(y)=\tau_{3}(y)=\tau_{4}(y)=0, \ \tau_{2}(y)=-f_{1}
\end{equation}
Let $Im\lambda\neq0$. Putting this general solution in (\ref{15})
yields
\begin{eqnarray}\label{20}
d_{11}=\frac{\Delta_{12}}{\omega(\lambda)}\int_{c+}^{b}\psi^{+}(y,\lambda)f(y)dy+\frac{\Delta_{12}f_{1}}{\omega(\lambda)},
\ \ d_{12}=0
\\
d_{21}=\frac{\Delta_{12}f_{1}}{\omega(\lambda)}, \ \
d_{22}=\frac{\Delta_{34}}{\omega(\lambda)}\int_{a}^{c-}\varphi^{-}(y,\lambda)f(y)dy
\end{eqnarray}
Thus the problem (\ref{14})-(\ref{15}) has an unique solution,

\begin{eqnarray}\label{22}
Y(x,\lambda)=\left\{\begin{array}{c}
               \frac{\Delta_{34}\psi^{-}(x,\lambda)}{\omega(\lambda)}\int_{a}^{x}\varphi^{-}(y,\lambda)f(y)dy +
\frac{\Delta_{34}\varphi^{-}(x,\lambda)}{\omega(\lambda)}\int_{x}^{c-}\psi^{-}(y,\lambda)f(y)dy \\ +\frac{\Delta_{12}\varphi^{-}(x,\lambda}{\omega(\lambda)})(\int_{c+}^{b}\psi^{+}(y,\lambda)f(y)dy+f_{1})\ , \ \ \ \ \ \ \ \ for \  x \in [a,c) \\
                \\
               \frac{\Delta_{12}\psi^{+}(x,\lambda)}{\omega(\lambda)}\int_{c+}^{x}\varphi^{+}(y,\lambda)f(y)dy +
\frac{\Delta_{12}\varphi^{+}(x,\lambda)}{\omega(\lambda)}\int_{x}^{b}\psi^{+}(y,\lambda)f(y)dy\\ +\frac{\Delta_{34}\psi^{+}(x,\lambda)}{\omega(\lambda)}\int_{a}^{c-}\varphi^{-}(y,\lambda)f(y)dy +\frac{\Delta_{12}f_{1}\varphi^{+}(x,\lambda)}{\omega(\lambda)}\ , \ \ \ \ \ \ \ \ for \ x \in (c,b] \\
             \end{array}\right.
\end{eqnarray}
Consequently
\begin{eqnarray}\label{2.16}
Y(x,\lambda)=Y_{0}(x,\lambda)+
f_{1}\Delta_{12}\frac{\varphi(x,\lambda)}{\omega(\lambda)}.
\end{eqnarray}
where $G_{0}(x,\lambda)$ and $Y_{0}(x,\lambda)$ is the same with
(\ref{23}) and (\ref{ye}) respectively. From the equalities
(\ref{10}) and (\ref{23})  it follows that
\begin{eqnarray}\label{2.17}
(G_{0}(x,.;\lambda))'_{\beta}=\frac{\varphi(x,\lambda)}{\omega(\lambda)}.
\end{eqnarray}
By using (\ref{ye}), (\ref{22}) and (\ref{2.17}) we deduce that
\begin{eqnarray}\label{2.18}
Y(x,\lambda)&=&\Delta_{34} \int_{a}^{c-} G_{0}(x,y;\lambda)f(y)dy+
\Delta_{12}\int_{c+}^{b} G_{0}(x,y;\lambda)f(y)dy \nonumber\\&+&
f_{1}\Delta_{12}(G_{0}(x,.;\lambda))'_{\beta}
\end{eqnarray}
Consequently,  the solution  $Y(F,\lambda)$ of the operator equation
(\ref{y1}) has the form
\begin{eqnarray}\label{2.19}
Y(F,\lambda)=(Y(x,\lambda), (Y(.,\lambda))'_{\beta})
\end{eqnarray}
From (\ref{2.18}) and (\ref{2.19}) it follows that
\begin{eqnarray}\label{2.20}
Y(F,\lambda)=(<G_{x,\lambda},\overline{F}>_{1},
(<G_{x,\lambda},\overline{F}>_{1})'_{\beta})
\end{eqnarray}
where under Green' s vector $G_{x,\lambda}$ we mean
\begin{eqnarray}\label{2.21}
G_{x,\lambda}:=(G_{0}({x,.;\lambda}),
(G_{0}({x,.;\lambda}))'_{\beta})
\end{eqnarray}
Now, making use (\ref{23}), (\ref{2.18}), (\ref{2.19}), (\ref{2.20})
and (\ref{2.21}) we see that if $\lambda$  not an eigenvalue of
operator A then
\begin{eqnarray}\label{2.22}
Y(F,\lambda)\in D(A) \ \textrm{for} \ \  F \in H_{1},
\end{eqnarray}
\begin{eqnarray}\label{2.23}
Y((\lambda-A)F,\lambda)=F, \ \textrm{for} \in D(A)
\end{eqnarray}
and
\begin{equation}\label{2.24} \| Y(F,\lambda) \| \leq
\ |Im\lambda|^{-1}\| F \| \ \textrm{for} \ F \in H_{1}, \ \ \ Im
\lambda \neq0 .
\end{equation}
Hence, each nonreal $\lambda\in \mathbb{C}$  is a regular point of
an operator A and
\begin{eqnarray}\label{2.25}
R(\lambda,A)F=(<G_{x,\lambda},\overline{F}>_{1},
(<G_{x,\lambda},\overline{F}>_{1})'_{\beta}) \ \textrm{for} \ \  F
\in H_{1}
\end{eqnarray}
Because of (\ref{2.22}) and (\ref{2.25})
\begin{eqnarray}\label{2.26}
(\lambda-A)D(A)=(\overline{\lambda}-A)D(A)=H_{1} \ \textrm{for} \
Im\lambda \neq0.
\end{eqnarray}
\begin{thm}
The Resolvent operator $R(\lambda,A)$ is compact in the Hilbert
space $H_{1}$.
\end{thm}
\begin{proof}
\end{proof}
\section{Selfadjoint realization of
the problem }
 At first we shall prove the following lemmas.
\begin{lem}\label{lem3.1}
The domain  $D(A)$ is dense in $H_{1}$.
\end{lem}
\begin{proof}

\end{proof}
\begin{lem}\label{2.1}
The linear operator $A$ is symmetric in the Hilbert space $H_{1}$.
\end{lem}
\begin{proof}
Let $F=(f(x), T'_{b}(f)),G=(G_{1}(x), T'_{b}(f)) \in D(A)$. By
partial integration we get
\begin{eqnarray}\label{2.2}
<AF,G>_{1}&=& \Delta_{12} \ \int_{a}^{c-} (\mathcal{L}
f)(x)\overline{g(x)}dx +\Delta_{34} \int_{c+}^{b} (\mathcal{L}
f)(x)\overline{g(x)}dx \nonumber\\&+&
\frac{\Delta_{34}}{\Delta_{0}}T_{b}(f)\overline{T'_{b}(g)}
 \nonumber\\&=&<F,AG>_{1}+ \Delta_{12} \ W(f, \overline{g};c-0) -
\Delta_{12} \ W(f, \overline{g};a)  \nonumber\\
&+&\Delta_{34} \ W(f,\overline{g};b) -
\Delta_{34} \ W(f,\overline{g};c+0) \nonumber\\
&+&
\frac{\Delta_{34}}{\Delta_{0}}(T'_{b}(f)\overline{T_{b}(g)}-T_{b}(f)\overline{T'_{b}(g)}).
\end{eqnarray} From the definition of domain D(A) wee see easily
that $W(f, \overline{g};a)=0$. The direct calculation gives
$$
T'_{b}(f)\overline{T_{b}(g)}-T_{b}(f)\overline{T'_{b}(g)})=-\Delta_{0}W(f,\overline{g};b)
\  \textrm{and} \ W(f,\overline{g};c-0) =\frac{
\Delta_{34}}{\Delta_{12}} \ W(f, \overline{g};c+0).
$$
Substituting these equalities in (\ref{2.2}) we have
\[<AF,G> _{1}\ = \ <F,AG>_{1} \  \textrm{for \ every} \  F,G \in D(A), \]so the operator A is
symmetric in $H$. The proof is complete.
\end{proof}
\begin{rem}\label{rem2}By Lemma \ref{2.1} all eigenvalues of the problem $(\ref{1})-(\ref{5})$ are
real. Therefore it is enough to investigate only real-valued
eigenfunctions.  Taking in view this fact, we can assume that the
eigenfunctions are real-valued.
\end{rem}
\begin{cor}\label{cor2}If $\lambda_{n}$ \textrm{and} $\lambda_{m}$ are distinct eigenvalues
of the problem $(\ref{1})-(\ref{5})$, then the corresponding
eigenfunctions $u_{n}(x)$ \ and \ $u_{m}(x)$ \ is orthogonal in the
sense of the following equality
\begin{eqnarray}\label{2.3}
\Delta_{12} \ \int_{a}^{c-} u(x)v(x)dx + \Delta_{34} \int_{c+}^{b}
u(x)v(x)dx + \frac{\Delta_{34}}{\Delta_{0} }T'_{b}(u)T'_{b}(v)=0.
\end{eqnarray}
\end{cor}
\begin{proof}
The proof is immediate from the fact that, the eigenelements\\
$(u(x), T'_{b}(u)) \ \textrm{and} \ (v(x), T'_{b}(v))$ of the
symmetric linear operator $A$ is orthogonal  in the Hilbert space
$H_{1}.$
\end{proof}
\begin{thm}\label{t5f}
The operator A is self-adjoint in $H_{1}$.
\end{thm}
\begin{proof}
\end{proof}

%-----------------------------------------------------------------

\begin{thebibliography}{99}
\bibitem{be}J. Ao, J. Sun and M. Zhang  {\em The finite spectrum of Sturm{-}Liouville problems
with transmission conditions,} \ Comput. Appl. Math., 218
1166{-}1173(2011).

\bibitem{ap} P. Appell,  {\em Sur l'\'{e}quation  $\frac{\partial^{2}z}{\partial x^{2}}-\frac{\partial z}{\partial y}=0$ et la th\'{e}orie de la chaleur,}
 \ J. Math. Pures Appl. 8 , 187{-}216(1892).

\bibitem{ba} E. Bairamov and E. U\u{g}urlu,
{\em The determinants of dissipative Sturm-Liouville operators with
transmission conditions,} Math. Comput. Modelling, Vol. 53, Nr. 5-6
, 805-813(2011).



\bibitem{bu}H. Burkhardt, {\em Sur les fonctions de Green relatives \`{a} un domaine d'une dimension}.
Bull. Soc. Math., 22, 71-75( 1894).



\bibitem{ch}B. Chanane,
{\em  Sturm-Liouville problems with impulse effects,} \  Appl. Math.
Comput., 190/1 pp. 610{-}626(2007).


\bibitem{green}G. Green,
{\em  An essay on the application of mathematical analysis to
theories of electricity and magnetism, } J. reine angewand. Math.,
39  pp. 73{-}89(1850).


\bibitem{os1} M. Kadakal and O. Sh. Mukhtarov, {\em Discontinuous
 Sturm-Liouville Problems Containing Eigenparameter in the Boundary
Conditions}. Acta Mathematica Sinica, English Series Sep., Vol. 22,
No. 5, pp. 1519-1528(2006).


\bibitem{ki}G. Kirchhoff ,
{\em  Zur Theorie der Lichtstrahlen, } Ann. Phys., 18  pp.
663{-}695(1883).


\bibitem{fu}C. T. Fulton, {\em Two-point boundary value problems with
eigenvalue parameter contained in the boundary conditions,} Proc.
Roy. Soc. Edin. 77A, 293{-}308(1977).


\bibitem{ho}E. W. Hobson, {\em Synthetical solutions in the conduction of heat,}
Proc. London Math. Soc. 19, 279{-}294(1887).


\bibitem{ne}C. Neumann, {\em Undersuchungen über das Logaritmische and Newton'sche Potential, }
Teubner, Leipzig, 1877.



\bibitem{ka1}O. Sh. Mukhtarov and H. Demir,
{\em Coersiveness of the discontinuous initial- boundary value
problem for parabolic equations}, Israel J. Math., Vol. 114,
239{-}252(1999).





\bibitem{ka2}O. Sh. Mukhtarov and  S. Yakubov,
{\em Problems for ordinary differential equations with transmission
conditions}, Appl. Anal., Vol. 81, 1033{-}1064(2002).


\bibitem{li}A. V. Likov and Y. A. Mikhailov, {\em  The heory of Heat and Mass
Transfer}, Qosenergaizdat,1963 (In Russian).


\bibitem{ti}E. C. Titchmarsh, {\em
Eigenfunctions Expansion Associated with Second Order Differential
Equations I}, second edn. Oxford Univ. Press, London, 1962.

\bibitem{tit}I. Titeux, Ya. Yakubov, {\em
Completeness of root functions for thermal condition in a strip with
piecewise continuous coefficients}, Math. Models Methods Appl. Sci.
Vol.7  1035–1050(1997).



%-----------------------------------------------------------------
\end{thebibliography}
\end{document}